\documentclass[11pt,a4paper,leqno]{article}

\usepackage{amsmath}
\usepackage{amsthm}
\usepackage{amsfonts}

\newtheorem{thm}{Theorem}
\newtheorem{lem}{Lemma}
\newtheorem{cor}{Corollary}
\theoremstyle{remark}
\newtheorem{rem}{Remark}

\newcommand{\R}{\mathbb{R}}
\newcommand{\Rp}{\mathbb{R}_+}
\newcommand{\Rm}{\mathbb{R}_-}
\renewcommand{\P}{\textbf{P}}
\newcommand{\E}{\textbf{E}}
\newcommand{\I}{\textbf{I}}
\newcommand{\SK}{\mathcal{S}^*}

\parskip \smallskipamount

\title{Random walks with long-tailed increments}
\author{Serguei Foss and Stan Zachary}
\date{\today}

\begin{document}

\begin{center}
  \bfseries THE MAXIMUM ON A RANDOM TIME INTERVAL OF A RANDOM WALK
  WITH LONG-TAILED INCREMENTS AND NEGATIVE DRIFT
\end{center}

\begin{center}
  \sc
  By Serguei Foss and Stan Zachary
  \footnotetext{
    {\it American Mathematical Society 1991 subject classifications.\/}
    Primary 60G70; secondary  60K30, 60K25

    {\it Key words and phrases.\/} ruin probability, long-tailed
    distribution, subexponential distribution.

    {\it Short title.\/}  Maximum on a random time interval.
    }
\end{center}

\begin{center}
  \textit{Heriot-Watt University}
\end{center}

\begin{quotation}\small
  We study the asymptotics for the maximum on a random time interval
  of a random walk with a long-tailed distribution of its increments
  and negative drift.  We extend to a general stopping time a result
  by Asmussen (1998), simplify its proof, and give some converses.
\end{quotation}

\section{Introduction}

Random walks with long-tailed increments have many important
applications in insurance, finance, queueing networks, storage
processes, and the study of extreme events in nature and elsewhere.
See, for example, Embrechts \textit{et al.}\ (1997), Asmussen (1998,
1999) and Greiner \textit{et al.\/} (1999) for some background.  In
this paper we study the distribution of the maximum of such a random
walk over a random time interval.

Let $F$ be the distribution function of the increments of a random
walk~$\{S_n\}_{n\ge0}$ with $S_0=0$.  Suppose that this distribution
has a finite negative mean and that $F$ is long-tailed in the positive
direction (see below for this and other definitions).  Of interest is
the asymptotic distribution of the maximum of $\{S_n\}$ over the
interval~$[0,\sigma]$ defined by some stopping time~$\sigma$.  Some
results for the case where $\sigma$ is independent of $\{S_n\}$ are
known (again see below).  However, relatively little is known for
other stopping times.  Asmussen (1998) gives the expected result for
the case $\sigma=\tau$, where
\begin{equation}
  \label{eq:taudef}
  \tau = \min\{n\ge 1: S_n\le0\}
\end{equation}
(see also Heath \textit{et al.\/} (1997) and Greiner \textit{et al.\/}
(1999)).  This result requires the further condition that the
distribution function~$F$ has a right tail which belongs to the
class~$\SK$ introduced by Kl\"uppelberg (1988) (we shall simply write
$F\in\SK$).  In the present paper we extend Asmussen's result to a
general stopping time~$\sigma$.  In doing so we also simplify the
derivation of the original result, and we show that the
condition~$F\in\SK$ is necessary as well as sufficient for it to hold.
We also give a useful characterisation of the class~$\SK$.  Finally,
as a corollary of our results, we give a probabilistic proof of the
known result that any distribution function~$G\in\SK$ is
subexponential.

Thus, let $\{\xi_n\}_{n\ge1}$ be a sequence of independent identically
distributed random variables with distribution function~$F$.  We
assume throughout that
\begin{equation}\label{eq:neg}\tag{NEG}
  \E\xi_n = -m < 0.
\end{equation}
We further assume throughout that the distribution function~$F$ is
\emph{long-tailed} (LT), that is, that
\begin{equation}
  \label{eq:lt}\tag{LT}
  \text{$\overline{F}(x)>0$ for all $x$},\qquad
  \lim_{x\to\infty}\frac{\overline{F}(x-h)}{\overline{F}(x)} = 1,
  \quad\text{for all fixed $h>0$}.
\end{equation}
Here, for any distribution function~$G$ on $\R$, $\overline{G}$
denotes the tail distribution given by $\overline{G}(x)=1-G(x)$.
Define the random walk $\{S_n\}_{n\ge0}$ by
$S_0=0$, $S_n=\sum_{i=1}^n\xi_i$ for $n\ge1$.
For $n\ge0$, let
\begin{math}
  M_n = \max_{0\leq i\leq n} S_i,
\end{math}
and let
\begin{math}
  M = \sup_{n\geq 0} S_n.
\end{math}
Similarly, for any stopping time $\sigma$ (with respect to any
filtration~$\{\mathcal{F}_n\}_{n\ge1}$ such that, for each $n$, $\xi_n$
is measurable with respect to $\mathcal{F}_n$ and $\xi_{n+1}$ is
independent of $\mathcal{F}_n$), let
\begin{math}
  M_\sigma = \max_{0\leq i\leq\sigma} S_i.
\end{math}
We are interested in the asymptotic distribution of $M_\sigma$ for a
general stopping time~$\sigma$ (which need not be a.s.\ finite).  In
particular we are interested in obtaining conditions under which
\begin{equation}
  \label{eq:Msigma}
  \lim_{x\to\infty}\frac{\P(M_\sigma>x)}{\overline{F}(x)} = \E\sigma.
\end{equation}

We require first some further definitions.  For any distribution
function~$G$ on $\R$ define the integrated, or \emph{second-tail},
distribution function~$G^s$ by
\begin{math}
  \overline{G^s}(x) = \min \bigl( 1,
    \int_x^{\infty} \overline{G}(t)\, dt
  \bigr).
\end{math}
A distribution function~$G$ on $\Rp$ is \emph{subexponential} if and
only if $\overline{G}(x)>0$ for all $x$ and
$\lim_{x\to\infty}\overline{G^{*2}}(x)/\overline{G}(x) = 2$
(where $G^{*2}$ is the convolution of $G$ with itself).  More
generally, a distribution function~$G$ on $\R$ is subexponential if
and only if $G^+$ is subexponential, where
$G^+=G\I_{\R_+}$ and $\I_{\R_+}$ is the indicator function of $\R_+$.
It is
known that the subexponentiality of a distribution depends only on its
(right) tail, and that a subexponential distribution is long-tailed.
When $F$ is subexponential, it is elementary that the
result~\eqref{eq:Msigma} holds for any a.s.\ constant $\sigma$.  (The
condition~\eqref{eq:neg} is not required here.  See, for example,
Embrechts \textit{et al.}\ (1997), or Sigman (1999).)  In the case
where $F^s$ is subexponential, the
asymptotic distribution of $M$ is known---in particular
$\P(M>x)=O(\overline{F^s}(x))$ as $x\to\infty$ (see Veraverbeke
(1977), Embrechts and Veraverbeke (1982), and, for a simpler
treatment, Embrechts \textit{et al.}\ (1997)).

A distribution function~$G$ on $\R$ belongs to the class~$\SK$ if and
only if $\overline{G}(x)>0$ for all $x$ and
\begin{equation}
  \label{eq:skdef}
  \int_0^x\overline{G}(x-y)\overline{G}(y)\,dy
  \sim 2m_{G^+}\overline{G}(x),\quad\text{as $x \to \infty$},
\end{equation}
where
\begin{math}
  m_G{^+} = \int_0^\infty \overline{G}(x) \,dx
\end{math}
is the mean of $G^+$.  It is again known that the property $G\in\SK$
depends only on the tail of $G$, and that if $G\in\SK$, then both $G$
and $G^s$ are subexponential---see Kl\"uppelberg (1988).


Consider first the case where the stopping time~$\sigma$ is
independent of the sequence~$\{\xi_n\}$.  Here, under the further
condition that the distribution function~$F$ is subexponential, the
result~\eqref{eq:Msigma} is well known to hold for any stopping
time~$\sigma$ such that
\begin{equation}
  \label{eq:expbdd}
  \E\exp\lambda\sigma<\infty, \quad\text{for some $\lambda>0$}.
\end{equation}
In this particular case the condition~\eqref{eq:neg} is not required
(see, for example, Embrechts \textit{et al.}\ (1997) and the
references therein.)  The condition~\eqref{eq:expbdd} may be dropped
by suitably strengthening the subexponentiality condition on $F$ (see
Borovkov and Borovkov (2001), Korshunov (2001)).


The first results for a stopping time~$\sigma$ which is not
independent of the sequence~$\{\xi_n\}$ are given by Heath \textit{et
  al.\/} (1997, Proposition~2.1) and, under more general conditions,
by Asmussen (1998, Theorem 2.1)---see also Greiner \textit{et al.\/}
(1999, Theorem~3.3).  Asmussen shows that if, in addition to our
present conditions~\eqref{eq:neg} and \eqref{eq:lt}, we have
$F\in\SK$, then the result~\eqref{eq:Msigma} holds with $\sigma=\tau$,
where $\tau$ is as given by \eqref{eq:taudef}.  (Asmussen omits to
state formally the necessity of some condition of the form $F\in\SK$.
However, this is rectified in the more recent paper by Asmussen
\textit{et al.}\ (2001).  See also Asmussen, Foss and Korshunov (2002)
for further extensions.)

The main result of the present paper is Theorem~\ref{r:ms}, which
shows that, again under the condition~$F\in\SK$, the
result~\eqref{eq:Msigma} holds for a general stopping time~$\sigma$.
The theorem also shows that, for a wide class of stopping
times~$\sigma$, including $\sigma\equiv\tau$, the condition~$F\in\SK$
is necessary as well as sufficient for this result.  In proving
Theorem~\ref{r:ms} we of necessity simplify the derivation of
Asmussen's original result, which was quite tricky (as was that of
Greiner \textit{et al.\/} (1999)).  The proof requires
Theorem~\ref{r:se3}, which gives a characterisation of the
class~$\SK$.  One half of this theorem is due to Asmussen \textit{et
  al.}\ (2001).  Finally, as already remarked and as a very simple
corollary of Theorem~\ref{r:ms}, we give a probabilistic proof of the
result referred to above that any $G\in\SK$ is subexponential.

\section{Results}
\label{sec:results}

We state first our main result, which is Theorem~\ref{r:ms}.  We give
also Corollaries~\ref{r:kl} and \ref{r:ind}.  We then proceed to the
proofs, which are
via Theorem~\ref{r:se3} and a sequence of lemmas.  A further necessary
lemma, which is a fairly routine application of some results from
renewal theory, is relegated to the Appendix.  As remarked above, the
proof of part~(i) of Theorem~\ref{r:ms}, in the case $\sigma=\tau$,
follows the general approach of Asmussen (1998) with some
simplification of the argument.  Recall that the
conditions~\eqref{eq:neg} and \eqref{eq:lt} are assumed to hold
throughout.

\begin{thm}\label{r:ms}~

  (i) Suppose that $F\in\SK$.  Let $\sigma\le\infty$ be any stopping
  time.  Then
  \begin{equation}
    \label{eq:sigma}
    \lim_{x\to\infty}\frac{\P(M_{\sigma}>x)}{\overline{F}(x)} = \E\sigma.
  \end{equation}

  (ii) Suppose that the condition~\eqref{eq:sigma} holds for some
  stopping time~$\sigma$ such that $\P(\sigma>0)>0$, $\P(S_\sigma\le
  0)=1$ and $\E\sigma<\infty$.  Then $F\in\SK$.
\end{thm}

\begin{rem}\label{r:rem1}
  Suppose again that the conditions of Theorem~\ref{r:ms}~(i) hold.
  In the case~$\E\sigma<\infty$, it follows from from that result (and
  the well-known result, from \eqref{eq:lt}, that
  $\overline{F}(x)=o(\overline{F^s}(x))$ as $x\to\infty$) that
  $\P(M_{\sigma}>x)=o(\overline{F^s}(x))$ as $x\to\infty$ (in contrast
  to the case $\sigma=\infty$ a.s.).  In the case $\sigma<\infty$ a.s.,
  $\E\sigma=\infty$, we have both $\overline{F}(x)=o(\P(M_{\sigma}>x))$ (by
  Theorem~\ref{r:ms}~(i)) and, again,
  $\P(M_{\sigma}>x)=o(\overline{F^s}(x))$, in each case as $x\to\infty$.
  We give a proof of the latter result in the Appendix.  In this case
  a rich variety of behaviour is possible.
\end{rem}

Corollary~\ref{r:kl} is due to Kl\"uppelberg (1988).  However, we use
the results of this paper to give a simple probabilistic proof.

\begin{cor}[Kl\"uppelberg]
  \label{r:kl}
  Suppose that a distribution function~$G$ on $\R$ belongs to $\SK$.
  Then $G$ is subexponential.
\end{cor}

\begin{cor}
  \label{r:ind}\label{r:rempos}
  Suppose that, under the conditions of Theorem~\ref{r:ms}~(i), the
  stopping time~$\sigma$ is additionally independent of the
  sequence~$\{\xi_n\}$.  Then also
  \begin{equation*}
    \lim_{x\to\infty}\frac{\P(S_{\sigma}>x)}{\overline{F}(x)} = \E\sigma.
  \end{equation*}
\end{cor}

\begin{rem}
  The results of both Theorem~\ref{r:ms}~(i) and Corollary~\ref{r:ind}
  continue to hold even when the condition~\eqref{eq:neg} is dropped,
  provided that the mean of the random variables~$\xi_n$
  remains finite and a further condition is imposed on the tail of the
  distribution of the stopping time~$\sigma$.  Since \eqref{eq:neg} is
  assumed throughout the body of the paper, we give the details as
  Corollary~\ref{r:pos} in the Appendix.
\end{rem}

For several of our results and their proofs we require a
function~$h:\Rp\to\Rp$, which, since $F$ is long-tailed, may be chosen
such that
\begin{gather}
  h(x) \le x/2 \quad\text{for all $x$}, \label{eq:h1}\\
  h(x) \to \infty \quad\text{as $x\to\infty$}, \label{eq:h2}\\
  \frac{\overline{F}(x-h(x))}{\overline{F}(x)} \to 1 \quad\text{as
    $x\to\infty$}, \label{eq:h3}\\
  \text{there exists $x_0$ with $h(x+t)\le h(x)+t$ for all
    $x\ge{}x_0$, $t\ge0$}. \label{eq:h4}
\end{gather}

Now let $\pi$ be the distribution of $M$ ($=\sup_{n\ge0}S_n$).
Theorem~\ref{r:se3} below is in part due to Asmussen \textit{et al.}\
(2001), and provides a useful characterisation of the class~$\SK$.  In
particular, by taking the function~$g$ in the statements (a) and (b)
of the theorem to be the indicator function of the interval~$[0,c]$, we
obtain the equivalence of the condition~$F\in\SK$ to a local limit
result for $\pi((x,x+c])$ for \emph{any}, and hence for \emph{all},
$c>0$.

Let $\mathcal{G}$ be the class of functions on $\Rp$ which are
directly Riemann integrable (see, for example, Feller (1971, p.\ 362))
and nonnegative.  Let $\mathcal{G}^*$ be the subclass of $\mathcal{G}$
consisting of those functions which are additionally bounded away from
zero on some interval of nonzero width.

\begin{thm}\label{r:se3}
  For any function~$g:\Rp\to\R$ consider the property
  \begin{equation}
    \label{eq:gllt}
    \lim_{x\to\infty}\frac{1}{\overline{F}(x)}
    \int_0^\infty \pi(x+dt)g(t)
    = \frac{1}{m}\int_0^\infty g(t)\,dt.
  \end{equation}
  \begin{itemize}
  \item[(a)] If $F\in\SK$, then \eqref{eq:gllt} holds for all directly
    Riemann integrable functions~$g$ on $\Rp$.
  \item[(b)] If \eqref{eq:gllt} holds for any given
    $g\in\mathcal{G}^*$, then $F\in\SK$.
  \end{itemize}
\end{thm}

\begin{proof}
  The result~(a) is Corollary~1 of Asmussen \textit{et al.}\ (2001).
  We prove (b).

  Note first that, for any $g\in\mathcal{G}$,
  \begin{equation}
    \label{eq:gllb}
    \liminf_{x\to\infty}\frac{1}{\overline{F}(x)}
    \int_0^\infty \pi(x+dt)g(t)
    \ge \frac{1}{m}\int_0^\infty g(t)\,dt.
  \end{equation}
  This follows routinely from the lower bound~\eqref{eq:lt3} on the
  distribution~$\pi$ of $M$ given by Lemma~\ref{r:lt} in the Appendix.
  (In particular we may initially take $g$ to be zero outside a finite
  interval, and then use Fatou's Lemma to obtain the general result.)

  Now suppose that \eqref{eq:gllt} holds with $g$ given by
  $g_0\in\mathcal{G}^*$.  Let $\I_A$ denote the indicator function of
  any $A\subset\Rp$.  Then, for all sufficiently small $c>0$, we can
  find $\varepsilon>0$, $b\ge0$, and $g_2\in\mathcal{G}$ such that
  $g_0\equiv g_1+g_2$ where
  $g_1\equiv\varepsilon\I_{[b,b+c]}\in\mathcal{G}^*$.  Now
  \begin{multline*}
    \limsup_{x\to\infty}\frac{1}{\overline{F}(x)}
    \int_0^\infty \pi(x+dt)g_1(t) \\
    \le
    \lim_{x\to\infty}\frac{1}{\overline{F}(x)}
    \int_0^\infty \pi(x+dt)g_0(t)
    -
    \liminf_{x\to\infty}\frac{1}{\overline{F}(x)}
    \int_0^\infty \pi(x+dt)g_2(t).
  \end{multline*}
  It therefore follows from \eqref{eq:gllt} with $g$ given by $g_0$
  and from \eqref{eq:gllb} with $g$ given by each of $g_1$ and $g_2$,
  that \eqref{eq:gllt} also holds with $g$ given by $g_1$.  Since $F$
  is long-tailed, the property \eqref{eq:gllt} is preserved under any
  finite shift of the function~$g$.  Thus, finally, we obtain that
  \eqref{eq:gllt} holds for all $g$ of the form $\I_{[0,c]}$ for all
  sufficiently small $c>0$, and so, by additivity, for all $c>0$.

  Now fix any $c>0$.  Let the sequences of random
  variables~$\{\psi_n\}_{n\ge1}$, $\{T_n\}_{n\ge1}$, the random
  variable~$\nu$ and the constant~$p$ be as defined in the Appendix.
  From \eqref{eq:gllt} with $g\equiv\I_{[0,c]}$, a variation of the
  argument at the end of the proof of Lemma~\ref{r:lt} gives
  \begin{align*}
    \frac{c}{m}
    & = \lim_{x\to\infty}\frac{\P(M\in(x,x+c])}{\overline{F}(x)} \\
    & \ge \P(\nu=2)\limsup_{x\to\infty}
    \frac{\P(T_2\in(x,x+c])}{\overline{F}(x)}
    + \liminf_{x\to\infty}\sum_{n\ge1,\,n\neq2} \P(\nu=n)
    \frac{\P(T_n\in(x,x+c])}{\overline{F}(x)}\\
    & \ge \P(\nu=2)\limsup_{x\to\infty}
    \frac{\P(T_2\in(x,x+c])}{\overline{F}(x)}
    + \sum_{n\ge1,\,n\neq2} \P(\nu=n)
    \liminf_{x\to\infty}\frac{\P(T_n\in(x,x+c])}{\overline{F}(x)},
  \end{align*}
  where the last line above follows from Fatou's lemma Thus, from the
  lower bounds given by \eqref{eq:lt1} and \eqref{eq:lt2} in the
  Appendix (and the calculation leading to \eqref{eq:lt3}),
  \begin{displaymath}
    \limsup_{x\to\infty}\frac{\P(T_2\in(x,x+c])}{\overline{F}(x)}
    \le \frac{2pc}{(1-p)m},
  \end{displaymath}
  and hence, again using \eqref{eq:lt2},
  \begin{equation}
    \label{eq:t2lim}
    \lim_{x\to\infty}\frac{\P(T_2\in(x,x+c])}{\overline{F}(x)}
    = \frac{2pc}{(1-p)m},
  \end{equation}
  (where, as usual, the above equation includes the assertion that the
  limit exists).  Now, for the function~$h$ defined above (in fact we
  do not require the condition~\eqref{eq:h4}) for this proof), it
  follows from \eqref{eq:h1} that
  \begin{displaymath}
    \P(\psi_1+\psi_2\in(x,x+c],\,\psi_1 \le h(x),\,\psi_2 \le h(x)) = 0.
  \end{displaymath}
  Thus, from \eqref{eq:h2}, \eqref{eq:h3} and \eqref{eq:lt1},
  \begin{displaymath}
    \lim_{x\to\infty}\frac{1}{\overline{F}(x)}
    \P(\psi_1+\psi_2\in(x,x+c],\,\psi_i\le h(x))
    = \frac{pc}{(1-p)m}, \qquad i=1,2,
  \end{displaymath}
  and so, from \eqref{eq:t2lim},
  \begin{equation}\label{eq:small}
    \P(\psi_1+\psi_2\in(x,x+c],\,\psi_1>h(x),\,\psi_2>h(x))
    = o(\overline{F}(x)), \qquad\text{as $x\to\infty$}.
  \end{equation}
  Now it is convenient here to take $h$ such that $x-2h(x)$ is an
  integer multiple $n(x)$ of $c$ for all $x$.  Let also $d=p/(1-p)m$.
  Then, as $x\to\infty$,
  \begin{align}
    \P(\psi_1+\psi_2\in(x,x+c],\,\psi_1>h(x),\,\psi_2>h(x))
    & \ge
    \int_{h(x)}^{x-h(x)} \P(\psi_1 \in dt)
    \P(\psi_2 \in (x-t,x-t+c]) \nonumber\\
    & \sim
    dc \int_{h(x)}^{x-h(x)}
    \P(\psi_1 \in dt ) \overline{F}(x-t) \label{eq:flt}\\
    & =
    dc \sum_{k=1}^{n(x)} \int_{h(x)+(k-1)c}^{h(x) + kc}
    \P(\psi_1 \in dt ) \overline{F}(x-t) \nonumber\\
    & =
    (1+o(1)) d^2c \int_{h(x)}^{x-h(x)} \overline{F}(t) \overline{F}(x-t)\, dt,
    \label{eq:flt2}
  \end{align}
  where \eqref{eq:flt} follows from \eqref{eq:lt1}, and
  \eqref{eq:flt2} follows from \eqref{eq:lt1} and \eqref{eq:lt}.
  Thus, using also \eqref{eq:small}, we obtain
  \begin{equation}\label{eq:small2}
    \int_{h(x)}^{x-h(x)} \overline{F}(t) \overline{F}(x-t)\, dt
    = o(\overline{F}(x)), \qquad\text{as $x\to\infty$}.
  \end{equation}
  From \eqref{eq:h3},
  \begin{math}
    \overline{F}(x-t) = (1+o(1))\overline{F}(x)
  \end{math}
  as $x\to\infty$, uniformly in $t\in[0,h(x)]$.  Hence, from
  \eqref{eq:h2},
  \begin{displaymath}
    \int_0^{h(x)} \overline{F}(t) \overline{F}(x-t)\, dt
    \sim \overline{F}(x)\int_0^{h(x)} \overline{F}(t)\, dt
    \sim m_{F^+}\overline{F}(x), \qquad\text{as $x\to\infty$},
  \end{displaymath}
  and so the condition~\eqref{eq:small2} is equivalent to the
  condition~\eqref{eq:skdef} that $F\in\SK$.
\end{proof}

\begin{rem}
  Note that a trivial extension of the above proof gives directly
  (i.e.\ without reference to part (a) of the theorem) the result that
  if \eqref{eq:gllt} holds for any given $g\in\mathcal{G}^*$ then it
  does so for all directly Riemann integrable functions~$g$.
\end{rem}

For any $x\ge0$, define the stopping time
\begin{displaymath}
  \mu(x) = \min\{n: S_n>x\}.
\end{displaymath}
(Thus $\{\mu(x)\le n\}=\{M_n>x\}$.)  For any stopping time~$\sigma$
and any $x\ge0$, define
\begin{align*}
  A_{\sigma,1}(x) & = \{\mu(x)\le\sigma,\, S_{\mu(x)-1}\le h(x)\}\\
  A_{\sigma,2}(x) & = \{\mu(x)\le\sigma,\, S_{\mu(x)-1} > h(x)\},
\end{align*}
where $h$ is as given by \eqref{eq:h1}--\eqref{eq:h4}.  Define also
\begin{displaymath}
  \delta_\sigma(x) = \sup_{y\geq x}
  \frac{\P(A_{\sigma,2}(y))}{\overline{F}(y)}.
\end{displaymath}

Lemma~\ref{r:as1} below is, in an obvious sense, very close to what we
require for the proof of Theorem~\ref{r:ms}.  The loose ends are tied
up by Lemmas~\ref{r:deltatau}, \ref{r:tausigma} and \ref{r:dsms},
while Lemma~\ref{r:d} is necessary for the proof of
Lemma~\ref{r:deltatau}.

\begin{lem}
  \label{r:as1}
  Let $\sigma$ be a stopping time such that $\E\sigma<\infty$.  Then
  \begin{displaymath}
    \lim_{x\to\infty}\frac{\P(A_{\sigma,1}(x))}{\overline{F}(x)}
    = \E\sigma.
  \end{displaymath}
\end{lem}
\begin{proof}
For any $x\ge0$,
\begin{align*}
  \P(A_{\sigma,1}(x))
  & = \sum_{n=0}^{\infty}\int_{-\infty}^{h(x)}
  \P(\sigma >n, M_n\leq x, S_n \in dy, S_{n+1} > x) \\
  & \leq \sum_{n=0}^{\infty}\int_{-\infty}^{h(x)}
  \P(\sigma >n, M_n\leq x, S_n \in dy) \overline{F}(x-h(x)) \\
  & \leq \sum_{n=0}^{\infty} \P(\sigma >n) \overline{F}(x-h(x)) \\
  & =  (1+o(1))\E\sigma \overline{F}(x)
  \quad\text{as $x\to\infty$}.
\end{align*}
We now establish the lower bound.  For any positive integer~$N$,
\begin{align*}
  \P(A_{\sigma,1}(x))
  & \geq \sum_{n=0}^N \int_{-h(x)}^{h(x)}
  \P(\sigma >n, M_n\leq h(x), S_n \in dy, S_{n+1} >x) \\
  & \geq \sum_{n=0}^N
  \P(\sigma >n, M_n\leq h(x), S_n \in [-h(x),h(x)]) \overline{F}(x+h(x)) \\
  & = (1+o(1)) \overline{F}(x) \sum_{n=0}^N \P(\sigma >n)
  \quad\text{as $x\to\infty$},
\end{align*}
by \eqref{eq:h2} and \eqref{eq:h3}.  Let $N\to\infty$ to obtain
\begin{displaymath}
  \P(A_{\sigma,1}(x)) \geq (1+o(1))\E\sigma \overline{F}(x)
  \quad\text{as $x\to\infty$}.
\end{displaymath}
\end{proof}

Now define
\begin{displaymath}
  m^- = \int_0^\infty F(-y)\,dy.
\end{displaymath}
For any $t,x>0$, define
\begin{displaymath}
  \mathcal{D}(-t,-x)
  = \E(\#\{n\ge0: S_{n}>-t,\,S_{n+1}\le -t,\,\min_{0\le k\le n}S_k>-t-x\})
\end{displaymath}
to be the expected number of downcrossings by $\{S_n\}$ of $-t$ before
the random walk first reaches $-t-x$.  Define also
$\mathcal{D}(-t)=\mathcal{D}(-t,-\infty)$.

\begin{lem}
  \label{r:d}
  \begin{displaymath}
    \lim_{t,x\to\infty}\mathcal{D}(-t,-x) = \frac{m^-}{m}.
  \end{displaymath}
\end{lem}
\begin{proof}
  Asmussen (1998, Lemma~2.4) shows that
  \begin{equation}
    \label{eq:dlim}
    \lim_{t\to\infty}\mathcal{D}(-t) = \frac{m^-}{m}.
  \end{equation}
  For completeness we repeat his proof: let $R$ given by
  $R(A)=\sum_{n=0}^\infty\P(S_n\in A)$ be the renewal measure
  associated with the random walk~$\{S_n\}$.  Then there exist
  constants $a_1$ and $a_2$ such that $R[y,y+x] \le a_1+a_2 x$ for all
  $x\ge0$ and for all $y$.  We have
  \begin{displaymath}
    \mathcal{D}(-t) = \int_{-t}^\infty R(dy)F(-t-y)
    = \int_0^\infty R(dz-t)F(-z).
  \end{displaymath}
  When $F$ is nonlattice, $R(dz-t)$ converges vaguely to Lebesgue
  measure with density~$1/m$ as $t\to\infty$.  It follows that
  \begin{displaymath}
    \lim_{t\to\infty}\mathcal{D}(-t)
    = \frac{1}{m}\int_0^\infty F(-z)\,dz
    = \frac{m^-}{m}.
  \end{displaymath}
  The usual straightforward modifications are required to establish
  also \eqref{eq:dlim} in the lattice case.

  Now further, there exists $K<\infty$ with
  \begin{equation}\label{eq:k}
    \sup_{t>0}\mathcal{D}(-t) = K.
  \end{equation}
  Thus, by the strong Markov property,
  \begin{displaymath}
    \mathcal{D}(-t,-x) \le \mathcal{D}(-t) \le \mathcal{D}(-t,-x) + K\P(M>x)
  \end{displaymath}
  and so the required result follows from \eqref{eq:dlim} and since
  $\P(M>x)\to0$ as $x\to\infty$.
\end{proof}

\begin{lem}\label{r:deltatau}
  $\lim_{x\to\infty}\delta_\tau(x) = 0$ if and only if $F\in\SK$.
\end{lem}

\begin{proof}
  The proof is in part an extended, and somewhat clarified, version
  of the argument of Asmussen (1998).
  Define the reflected random walk (workload process or Lindley
  queueing theory recursion) $\{W_n\}_{n\ge0}$ by
  \begin{displaymath}
    W_0=0, \qquad W_n = \max(0,W_{n-1}+\xi_{n}), \quad n\ge 1.
  \end{displaymath}
  Note that, from \eqref{eq:neg}, this is an ergodic Markov chain.
  Note also that $W_n=S_n$ for $n<\tau$ (where $\tau$ is as defined by
  \eqref{eq:taudef}).  It is further convenient to extend the sequence
  $\{\xi_n\}_{n\ge1}$ to the doubly-infinite sequence
  $\{\xi_n\}_{-\infty<n<\infty}$ of independent identically
  distributed random variables with distribution function~$F$, and to
  define also the stationary version~$\{W^n\}_{-\infty<n<\infty}$ of
  the above workload process (indexed on $(-\infty,\infty)$) by
  \begin{equation}
    \label{eq:wsdef}
    W^n = \max(0,\,\sup_{j\geq 0} \sum_{i=0}^{-j} \xi_{n+i})
  \end{equation}
  Note that $W^n=\max(0,W^{n-1}+\xi_{n})$ for all $n$.  It is well
  known that the common distribution of the $W^n$ is given by $\pi$
  (as is clear from \eqref{eq:wsdef}).

  For any $x>0$, let
  \begin{displaymath}
    N^-(x) = \#\{n: 1\le n<\tau,\, S_n>x,\,S_{n+1}\le x\}
  \end{displaymath}
  be the number of downcrossings of $x$ in $[0,\tau]$ by $\{S_n\}$ or
  $\{W_n\}$.
  Note that, by the ergodicity of the process~$\{W^n\}$,
  \begin{align*}
    \frac{\E N^-(x)}{\E \tau}
    & = \P(W^0>x,\,W^0+\xi_1\le x)\\
    & = \int_0^\infty\pi(x+dt)F(-t).
  \end{align*}
  From \eqref{eq:neg}, the function~$g$ on $\Rp$ defined by
  $g(t)=F(-t)$ belongs to $\mathcal{G}^*$.  Hence, by
  Theorem~\ref{r:se3},
  \begin{equation}
    \label{eq:en}
    \E N^-(x) \sim \frac{m^-}{m}
    \E\tau \overline{F}(x)
    \text{ as $x\to\infty$ \quad if and only if \quad}
    F\in\SK.
  \end{equation}
  We also have
  \begin{align*}
    \E[N^-(x)\I(A_{\tau,1}(x))]
    & =
    \E[\I(A_{\tau,1}(x)) \E\{ N^-(x) ~|~ S_{\mu(x)}\}] \\
    & =
    \E[\I(A_{\tau,1}(x)) \mathcal{D}(-(S_{\mu(x)}-x), -x)].
  \end{align*}
  Since, for any $u\in (0,h(x))$,
  \begin{align*}
    \P(S_{\mu (x)}-x > h(x) ~|~ A_{\tau ,1}(x), S_{\mu (x) -1} \in du )
    & \geq \frac{\overline{F}(x+h(x))}{\overline{F}(x-h(x))} \\
    & \to 1, \quad\text{ as $x\to\infty$},
  \end{align*}
  by \eqref{eq:h3}, the overshoot $S_{\mu (x)} -x$ converges in
  distribution to $\infty$ as $x\to\infty$. Hence, again as
  $x\to\infty$,
  \begin{align}
    \E[N^-(x)\I(A_{\tau,1}(x))]
    & \sim \P(A_{\tau,1}(x))\frac{m^-}{m} \nonumber\\
    & \sim \frac{m^-}{m}\E\tau\overline{F}(x). \label{eq:en1}
  \end{align}
  Here, the first line above follows by Lemma~\ref{r:d} (and the
  spatial homogeneity of the random walk~$\{S_n\}$) and the second
  follows from Lemma~\ref{r:as1}.  It now follows from \eqref{eq:en}
  and \eqref{eq:en1} that
  \begin{displaymath}
    \E[N^-(x)\I(A_{\tau,2}(x))]
    = o(\overline{F}(x))
    \text{ as $x\to\infty$ \quad if and only if \quad}
    F\in\SK.
  \end{displaymath}

  Since also $1 \le \E(N^-(x)\,|\,A_{\tau,2}(x)) \le K$, where $K$ is
  as defined by \eqref{eq:k}, the required result now follows.
\end{proof}

\begin{lem}\label{r:tausigma}
  ~

  (i) Suppose that $\lim_{x\to\infty}\delta_\tau(x)=0$.  Let $\sigma$
  be any stopping time such that $\E\sigma<\infty$.  Then
  $\lim_{x\to\infty}\delta_\sigma(x)=0$.

  (ii) Suppose that there exists a stopping time~$\sigma$ such that
  $\P(\sigma>0)>0$, $\P(S_\sigma\le 0)=1$ and
  $\lim_{x\to\infty}\delta_\sigma(x)=0$.  Then
  $\lim_{x\to\infty}\delta_\tau(x)=0$.
\end{lem}
\begin{proof}
  To prove (i), note first that it follows from \eqref{eq:h4} (where
  $x_0$ is as defined there) that, for all $x\geq x_0$ and for all
  $t\ge0$,
  \begin{align}
    \P(\mu(x+t)\le\tau, \,S_{\mu(x+t)-1}>h(x)+t)
    & \le \P(\mu(x+t)\le\tau, \,S_{\mu(x+t)-1}>h(x+t)) \nonumber\\
    & \le \delta_\tau(x)\overline{F}(x). \label{eq:muxt}
  \end{align}
  Now define the sequence of stopping times~$\{\tau_k\}_{k\ge0}$ by
  \begin{equation}
    \label{eq:taukdef}
    \tau_0=0, \qquad
    \tau_k=\min\{n:n>\tau_{k-1},\,S_n\le{}S_{\tau_{k-1}}\}, ~~ k\ge1,
  \end{equation}
  so that $\tau_k$ is the $k$th decreasing ladder time (and in
  particular $\tau_1=\tau$).  For all $k$, since $S_{\tau_k}\le0$, it
  follows from \eqref{eq:muxt} and the temporal and spatial
  homogeneity of the random walk~$\{S_n\}$ that
  \begin{displaymath}
    \P(\tau_k < \mu(x) \leq \tau_{k+1}, S_{\mu(x)-1}>h(x)~|~\sigma > \tau_k)
    \leq \delta_\tau(x) \overline{F}(x),
    \qquad\text{for all $x\ge x_0$}.
  \end{displaymath}
  Since also $\sigma$ is a.s.\ finite, it follows that, for any
  $x\ge{}x_0$,
  \begin{align*}
    \P(A_{\sigma,2}(x))
    & = \sum_{k\ge0}\P(\tau_k<\mu(x)\le\tau_{k+1},\,
    \sigma\ge\mu(x),\, S_{\mu(x)-1} > h(x))\\
    & \le \sum_{k\ge0}\P(\tau_k<\mu(x)\le\tau_{k+1},\,
    \sigma>\tau_k,\, S_{\mu(x)-1} > h(x))\\
    & = \sum_{k\ge0}\P(\sigma>\tau_k)\P(\tau_k<\mu(x)\le\tau_{k+1},\,
    S_{\mu(x)-1} > h(x) \,|\, \sigma>\tau_k)\\
    & \le \delta_\tau(x) \overline{F}(x)\sum_{k\ge0}\P(\sigma>\tau_k)\\
    & \le \delta_\tau(x) \overline{F}(x)\sum_{k\ge0}\P(\sigma>k)\\
    & = \delta_\tau(x) \overline{F}(x)\E\sigma,
  \end{align*}
  so that (i) now follows.

  To prove (ii) note first that we may assume, without loss of
  generality, that $\P(\sigma>0)=1$ (for otherwise we may simply
  condition on the event $\{\sigma>0\}$, which is assumed to have a
  nonzero probability).  The given conditions on $\sigma$ then imply
  that $\tau\le\sigma$ a.s..  Thus, for all $x>0$,
  $\P(A_{\tau,2}(x))\le\P(A_{\sigma,2}(x))$ and so the result follows
  immediately.

\end{proof}

\begin{lem}
  \label{r:dsms}
  Let $\sigma$ be any stopping time such that $\E\sigma<\infty$.  Then
  \begin{displaymath}
    \lim_{x\to\infty}\delta_\sigma(x) = 0
    \quad \text{ if and only if } \quad
    \lim_{x\to\infty}\frac{\P(M_{\sigma}>x)}{\overline{F}(x)} = \E\sigma.
  \end{displaymath}
\end{lem}
\begin{proof}
  We have $\P(M_{\sigma}>x)=\P(A_{\sigma,1}(x))+\P(A_{\sigma,2}(x))$, so
  that the result is immediate from Lemma~\ref{r:as1}.
\end{proof}

\begin{proof}[Proof of Theorem~\ref{r:ms}]
  In the case $\E(\sigma)<\infty$, the proofs of both (i) and (ii) are
  immediate from Lemmas~\ref{r:deltatau}, \ref{r:tausigma} and
  \ref{r:dsms}.  The extension to the case $\E(\sigma)=\infty$ is by a
  simple truncation argument.
\end{proof}

\begin{proof}[Proof of Corollary~\ref{r:kl}]
  Given $G\in\SK$, let $\{\phi_n\}_{n\ge1}$ be a sequence of i.i.d.
  random variables with distribution function~$G^+$.  Take the
  sequence~$\{\xi_n\}_{n\ge1}$ of the present paper to be given by
  $\xi_n=\phi_n-b$ for all $n$, where $b$ is chosen sufficiently large
  that these random variables each have a negative mean.  Since the
  property that a distribution belongs to $\SK$ is easily shown to be
  shift invariant, the common distribution function $F$ of the random
  variables $\xi_n$ belongs to $\SK$.  Thus also $F$ is long-tailed.

  To show that $G$ is subexponential we apply Theorem~\ref{r:ms} with
  $\sigma\equiv2$ to obtain
  \begin{align}
    \P(\phi_1 + \phi_2 > x)
    & = \P(\xi_1 + \xi_2 > x - 2b) \nonumber \\
    & \le \P(M_2 > x - 2b) \nonumber \\
    & \sim 2 \overline{F}(x - 2b) \nonumber \\
    & \sim 2 \overline{G^+}(x),
    \qquad\text{as $x\to\infty$}, \label{eq:gse1}
  \end{align}
  where the last line above follows since $G\in\SK$ implies that $G^+$
  is long-tailed.  Further,
  \begin{align}
    \P(\phi_1 + \phi_2 > x)
    & \ge \P(\phi_1>x) + \P(\phi_2>x) - \P(\phi_1>x, \phi_2>x) \nonumber\\
    & = 2\overline{G^+}(x) - (\overline{G^+}(x))^2 \nonumber \\
    & \sim 2 \overline{G^+}(x),
    \qquad\text{as $x\to\infty$}. \label{eq:gse2}
  \end{align}
  Hence, from \eqref{eq:gse1} and \eqref{eq:gse2}, $G$ is
  subexponential as required.
\end{proof}

\begin{proof}[Proof of Corollary~\ref{r:ind}]
  Under the conditions of Theorem~\ref{r:ms}~(i), it follows from
  Corollary~\ref{r:kl} that $F$ is subexponential.  Hence, for any
  $n\ge1$, $\P(S_n>x)\sim~n\overline{F}(x)$ as $x\to\infty$.  If the
  stopping time~$\sigma$ is independent of the sequence~$\{\xi_n\}$ a
  simple truncation argument now gives
  \begin{equation*}
    \liminf_{x\to\infty}\frac{\P(S_{\sigma}>x)}{\overline{F}(x)} \ge \E\sigma.
  \end{equation*}
  Since $\P(S_\sigma>x)\le\P(M_\sigma>x)$ for all $x$, the result now
  follows from Theorem~\ref{r:ms}~(i).
\end{proof}

\begin{rem}
  Theorem~\ref{r:ms} can also be used to give a (rather circuitous)
  probabilistic proof of the result that if $G\in\SK$ then $G^s$ is
  subexponential.  By taking $\sigma\equiv\tau$ in the theorem, and
  using standard renewal theory, we may show that, for the shifted
  version~$F$ of $G$, defined as in the above proof,
  \begin{displaymath}
    \lim_{x\to\infty}\frac{1}{\overline{F^s}(x)}\P(M>x)
    = \frac{1}{m}.
  \end{displaymath}
  That $F^s$, and so $G^s$, is subexponential now follows from the
  converse to Veraverbeke's Theorem proved by Korshunov (1997).
\end{rem}

\section*{Acknowledgement}

The authors are grateful to Claudia Kl\"uppelberg for some very
helpful discussions, to Denis Denisov for a very valuable
simplification of the proof of Lemma~\ref{r:as1}, and to the referee
for a very thorough reading of the paper and suggested improvements.

\appendix

\section*{Appendix}

We give here some auxiliary results relating to the successive ladder
heights and to the maximum of the process~$\{S_n\}$.  We also prove
the last statement of Remark~\ref{r:rem1}.

Define $\eta=\min\{n\ge 1: S_n>0\} \le \infty$, and let
\begin{equation}
  \label{eq:pdef}
  p = \P(\eta = \infty) = \P(M=0).
\end{equation}
Note that $0<p<1$.  Let $\{\psi_n\}_{n\ge1}$ be a sequence of i.i.d.\
copies of a positive random variable~$\psi$ such that, for all
(measurable) $B\subseteq\Rp$,
\begin{displaymath}
  \P(\psi\in B) = \P(S_\eta\in B \,|\, \eta<\infty).
\end{displaymath}
Let $\nu$ be a random variable, independent of the above sequence,
such that
\begin{displaymath}
  \P(\nu=n) = p(1-p)^n, \qquad n=0,1,2,\dots.
\end{displaymath}
Then it is a standard result that
$$
M =_D \sum_{i=1}^{\nu} \psi_i
$$
(here $\sum_1^0 = 0$ by definition).
For each $n\ge1$, define also
\begin{equation}
  \label{eq:tdef}
  T_n = \sum_{i=1}^{n} \psi_i.
\end{equation}

\begin{lem}
  \label{r:lt}
  Under the conditions~\eqref{eq:neg} and \eqref{eq:lt}, for all $c>0$
  in the case where $F$ is nonlattice, and for all positive
  multiples~$c$ of the span in the case where $F$ is lattice,
  \begin{align}
    \lim_{x\to\infty} \frac{\P(\psi\in(x,x+c])}{\overline{F}(x)}
    & = \frac{pc}{(1-p)m}, \label{eq:lt1}\\[1ex]
    \liminf_{x\to\infty} \frac{\P(T_n\in(x,x+c])}{\overline{F}(x)}
    & \ge \frac{npc}{(1-p)m}, \qquad n \ge 2, \label{eq:lt2}\\[1ex]
    \liminf_{x\to\infty} \frac{\P(M\in(x,x+c])}{\overline{F}(x)}
    & \ge \frac{c}{m}. \label{eq:lt3}
  \end{align}
\end{lem}

\begin{proof}

The results are reasonably well known.  However, the lemma is adapted
to the needs of the present paper, and, for completeness, we give
also proofs here.  We restrict ourselves to the nonlattice case.

Let $\hat{H}$ denote the ``taboo'' renewal measure on
$\Rm=(-\infty,0]$ given by, for $B\subseteq\Rm$,
\begin{displaymath}
  \hat{H}(B) = \sum_{n=0}^{\infty}\P(S_n \in B, M_n =0).
\end{displaymath}
Note that
\begin{equation}\label{eq:hh1}
  \lim_{x\to\infty} \hat{H}((-x,-x+c]) = \frac{pc}{m}, \qquad c>0,
\end{equation}
and hence there exist finite positive constants $a$, $b$ such that
\begin{equation}\label{eq:hh2}
  \hat{H}((-x,0]) \le ax+b, \qquad x\ge 0.
\end{equation}
Then, for any $c>0$ and all $x\ge0$,
\begin{equation}\label{eq:psi}
  \P(\psi \in (x,x+c]) = \frac{1}{1-p}
  \int_0^{\infty} \hat{H}(-dt)
  \left[\overline{F}(x+t)-\overline{F}(x+t+c)\right].
\end{equation}
Since $F$ satisfies \eqref{eq:lt}, we can choose a function
$h:\Rp\to\Rp$ satisfying the earlier conditions
\eqref{eq:h1}--\eqref{eq:h3} and such that
\begin{displaymath}
  h(x)\left[\overline{F}(x)
    - \overline{F}(x+h(x)+c)\right] = o(\overline{F}(x)),
  \qquad\text{as $x\to\infty$}.
\end{displaymath}
Then, from \eqref{eq:hh2},
\begin{align}
  0 & \le \int_0^{h(x)} \hat{H}(-dt)
  \left[\overline{F}(x+t)-\overline{F}(x+t+c)\right] \nonumber\\
  & \le \left(ah(x) +b\right)
  \left[\overline{F}(x) - \overline{F}(x+h(x)+c)\right] \nonumber\\
  & = o(\overline{F}(x))
  \qquad\text{as $x\to\infty$}. \label{eq:psia}
\end{align}
Further
\begin{align}
  \frac{p}{m} \int_{h(x)}^{\infty}
  \left[\overline{F}(x+t)-\overline{F}(x+t+c)\right]\,dt
  & = \frac{p}{m} \int_{h(x)}^{h(x)+c}\overline{F}(x+t)\,dt \nonumber\\
  & \sim \frac{pc}{m} \overline{F}(x)
  \qquad\text{as $x\to\infty$}, \label{eq:psib}
\end{align}
since $F$ satisfies \eqref{eq:lt} and by the condition~\eqref{eq:h3}
on $h$.  Finally, it follows from \eqref{eq:hh1} that, given
$\varepsilon>0$, for all sufficiently large $x$, and hence $h(x)$,
\begin{align}
  \left\lvert
    \int_{h(x)}^{\infty}
    \left(\hat{H}(-dt) - \frac{p}{m}\,dt\right)
    \left[\overline{F}(x+t)-\overline{F}(x+t+c)\right]
  \right\rvert \hspace{-10em}\nonumber\\
  & \le \varepsilon \sum_{k=0}^\infty
  \left[\overline{F}(x+h(x)+kc)-\overline{F}(x+h(x)+(k+2)c)\right] \nonumber\\
  & \le 2\varepsilon\overline{F}(x). \label{eq:psic}
\end{align}
The result \eqref{eq:lt1} now follows from
\eqref{eq:psi}--\eqref{eq:psic}.

To show \eqref{eq:lt2} note that, from \eqref{eq:h1}--\eqref{eq:h3},
\begin{align*}
  \P (T_2 \in (x,x+c])
  & \geq \P (T_2 \in (x,x+c], \psi_1 \leq h(x)) +
  \P (T_2 \in (x,x+c], \psi_2 \leq h(x))\\
  & \sim 2\P (\psi \in (x,x+c]),  \qquad\text{as $x\to\infty$}.
\end{align*}
Hence the result~\eqref{eq:lt2} follows for $n=2$ from \eqref{eq:lt1};
the result for
general $n$ now follows by induction arguments. Finally, the
result~\eqref{eq:lt3} follows from Fatou's lemma and \eqref{eq:lt2}, since
\begin{align*}
  \liminf_{x\to\infty}\frac{\P(M\in(x,x+c])}{\overline{F}(x)}
  & = \liminf_{x\to\infty}\sum_{n\ge1}\P(\nu=n)
  \frac{\P(T_n\in(x,x+c])}{\overline{F}(x)}\\
  & \ge \sum_{n\ge1}\P(\nu=n)
  \liminf_{x\to\infty}\frac{\P(T_n\in(x,x+c])}{\overline{F}(x)}\\
  & \ge \frac{c}{m}.
\end{align*}
\end{proof}

We now prove the claim made in Remark~\ref{r:rem1}.  Suppose that
the conditions of Theorem~\ref{r:ms}~(i) hold and that $\sigma<\infty$
a.s.  Given $\varepsilon>0$, we can find a positive integer $K$ and
$L>0$ such that $\P(S_K>-L,\,\sigma\le{}K)\ge1-\varepsilon$.  Then
\begin{align}
  \P(M>x,\,\sigma\le{}K)
  & \ge \P(S_K>-L,\,M>x,\,\sigma\le{}K) \nonumber\\
  & \ge (1-\varepsilon)\P(M>x+L) \label{eq:shift}\\
  & = (1+o(1))(1-\varepsilon)\P(M>x) \qquad\text{as $x\to\infty$}, \nonumber
\end{align}
where \eqref{eq:shift} above follows since the distribution of $M$ is
invariant under the obvious shift operation.

From the above, Theorem~\ref{r:ms}~(i), and Veraverbeke's Theorem
(Veraverbeke (1977)),
\begin{align*}
  \P(M_\sigma>x) & \le \P(M_K>x) + \P(M>x,\,\sigma>K) \\
  & \le (1+o(1))(K\overline{F}(x)+\frac{\varepsilon}{m}\overline{F^s}(x))
  \qquad\text{as $x\to\infty$}.
\end{align*}
Hence, since $\overline{F}(x)=o(\overline{F^s}(x))$ as $x\to\infty$,
we have
\begin{math}
  \limsup_{x\to\infty}\P(M_\sigma>x)/\overline{F^s}(x) \le \varepsilon/m.
\end{math}
Now let $\varepsilon\to0$ to obtain the required result.

Finally, in the following further corollary to Theorem~\ref{r:ms}, we
consider the extent to which Theorem~\ref{r:ms}~(i) and
Corollary~\ref{r:ind} remain true when the condition~\eqref{eq:neg} is
dropped.

\begin{cor}\label{r:pos}
  Suppose that $F\in\SK$ (which, by Corollary~\ref{r:kl} implies that
  $F$ is subexponential and so satifies~\eqref{eq:lt}), that the
  corresponding distribution of the random variables~$\xi_n$ has a
  finite mean, but that \eqref{eq:neg} does not necessarily hold.  Let
  $\sigma<\infty$ be a stopping time such that, for some function~$h$
  satisfying~\eqref{eq:h3},
  \begin{equation}
    \label{eq:pcond}
    \P(\sigma > h(x)) = o(\overline{F}(x)) \qquad\text{as $x\to\infty$}.
  \end{equation}
  Then, again,
  \begin{equation}
    \label{eq:psigma}
    \lim_{x\to\infty}\frac{\P(M_{\sigma}>x)}{\overline{F}(x)} = \E\sigma.
  \end{equation}
  If, additionally, $\sigma$ is independent of the
  sequence~$\{\xi_n\}$, then~\eqref{eq:psigma} also holds with
  $M_{\sigma}$ replaced by $S_{\sigma}$.
\end{cor}

\begin{proof}
  Choose any $a>\E\xi$.  For each $n\ge0$, define $\tilde{S}_0=0$,
  $\tilde{S}_n=\sum_{i=1}^n(\xi_i-a)$ (with $\tilde{S}_0=0$) and
  $\tilde{M}_n=\max_{0\le{}i\le{}n}\tilde{S}_i$.  Then, for each $x$,
  \begin{align}
    \P(M_{\sigma} > x)
    & \le \P(\tilde{M}_{\sigma} + a\sigma > x) \nonumber \\
    & \le \P(a\sigma > ah(x)) + \P(\tilde{M}_{\sigma} > x-ah(x)) \nonumber \\
    & \sim \E\sigma \overline{F}(x-ah(x)+a)
    \qquad\text{as $x\to\infty$} \label{eq:p2}\\
    & \sim \E\sigma \overline{F}(x)
    \qquad\text{as $x\to\infty$} \label{eq:p3}.
  \end{align}
  where \eqref{eq:p2} follows from Theorem~\ref{r:ms}~(i) and
  \eqref{eq:pcond}, while \eqref{eq:p3} follows since we may clearly
  replace $h(x)$ by $ah(x)-a$ in \eqref{eq:h3}.  Further it follows from
  Lemma~\ref{r:as1}, which clearly does not require the
  condition~\eqref{eq:neg}, that
  $\liminf\P(M_{\sigma}>x)/\overline{F}(x)\ge\E\sigma$.  Hence the
  result~\eqref{eq:psigma} is established.  The final assertion
  follows as in the proof of Corollary~\ref{r:ind}.
\end{proof}

Note that the condition~\eqref{eq:pcond} need not be unduly
restrictive.  For example, in the regularly varying case
$\overline{F}(x)=x^\alpha{}L(x)$ for some function~$L$ which is slowly
varying at infinity and some $\alpha<-1$ (for a finite mean), for a
function $h$ to satisfy~\eqref{eq:h3} it is sufficient that
$h(x)=o(x)$ as $x\to\infty$.  Hence the tail of the stopping
time~$\sigma$ need only be slightly lighter than that of the random
variables~$\xi_n$.



\begin{center}
\begin{minipage}[t]{7cm}
\begin{flushleft}
\sc
S. Foss\\
Department of Actuarial\\ \hspace*{0.5cm}Mathematics and Statistics\\
Heriot-Watt University\\
Riccarton\\
Edinburgh\\
UK, EH14 4AS\\
Email: \rm
s.foss@ma.hw.ac.uk
\end{flushleft}
\end{minipage}
\hfill
\begin{minipage}[t]{6cm}
\begin{flushleft}
\sc
S. Zachary\\
Department of Actuarial\\ \hspace*{0.5cm}Mathematics and Statistics\\
Heriot-Watt University\\
Riccarton\\
Edinburgh\\
UK, EH14 4AS\\
Email: \rm
s.zachary@ma.hw.ac.uk
\end{flushleft}
\end{minipage}
\end{center}


\newpage

\begin{thebibliography}{99}

\enlargethispage*{4 ex}


\bibitem{SA3} \textsc{Asmussen, S.} (1998).  Subexponential
  asymptotics for stochastic processes: extremal behavior, stationary
  distributions and first passage probabilities. {\it Ann.\ Appl.\
  Probab.}, {\bf 8}, 354--374.

\bibitem{SA2} \textsc{Asmussen, S.} (1999). Semi-Markov
  queues with heavy tails. In {\it Semi-Markov Models and
    Applications}, eds J.~Janssen and N~Limnios.  Kluwer.



\bibitem{AFK} \textsc{Asmussen, S.}, \textsc{Foss, S.G.}, and
  \textsc{Korshunov, D.A.} (2002).  Asymptotics for sums of
  random variables with local subexponential behaviour.  To appear.

\bibitem{AKKKT} \textsc{Asmussen, S.}, \textsc{Kalashnikov, V.},
  \textsc{Konstantinides, D.}, \textsc{Kl\"uppelberg, C.} and
  \textsc{Tsitiashvili, G.} (2001).  A local limit theorem for random
  walk maxima with heavy tails.  To appear.

\bibitem{BB} \textsc{Borovkov, A.A.} and \textsc{Borovkov, K.A.}
  (2001).  On large deviation probabilities for random walks.
  I.~Regularly varying distribution tails.  II.~Regularly exponential
  distribution tails.  \textit{Theory Probab.\ Appl.}, \textbf{46}, 2,
  209--232.  (In Russian.)

\bibitem{EKM} \textsc{Embrechts, P.}, \textsc{Kl\"uppelberg, C.}, and
  \textsc{Mikosch, T.} (1997).  {\it Modelling Extremal Events.}
  Springer Verlag.

\bibitem{EV} \textsc{Embrechts, P.} and \textsc{Veraverbeke, N.}
  (1982).  Estimates for the probability of ruin with special emphasis
  on the possibility of large claims.  \textit{Insurance Mathematics.\
  Econom.}, \textbf{1}, 55--72.

\bibitem{Fel} \textsc{Feller, W.} (1971).  {\it An Introduction to
    Probability Theory and its Applications.}  Wiley.

\bibitem{GJK} \textsc{Greiner, M.}, \textsc{Jobmann, M.} and
  \textsc{Kl\"uppelberg, C.} (1999).  Telecommunication traffic,
  queueing models, and subexponential distributions.  {\it Queueing
    Systems}, {\bf 33}, 125--152.

\bibitem{HRS} \textsc{Heath, D.},  \textsc{Resnick, S.} and
  \textsc{Samorodnitsky, G.} (1997).  Patterns of buffer overflow in a
  class of queues with long memory in the input stream.  {\it Ann.\ Appl.\
  Probab.}, {\bf 7}, 1021--1057.


\bibitem{Klu88} \textsc{Kl\"uppelberg, C.} (1988).  Subexponential
  distributions and integrated tails.  \textit{J.\ Appl.\ Prob.},
  \textbf{35}, 325--347.

\bibitem{Kor97} \textsc{Korshunov, D.A.} (1997).  On distribution tail
  of the maximum of a random walk.  {\it Stoch.\ Proc.\ Appl.},
  \textbf{72}, 1, 97--103.

\bibitem{Kor01} \textsc{Korshunov, D.A.} (2001).  Large-deviation
  probabilities for maxima of sums of independent random variables
  with negative mean and subexponential distribution.  \textit{Theory
    Probab.\  Appl.}, \textbf{46}, 2, 387--397.  (In Russian.)

\bibitem{Sig} \textsc{Sigman, K.} (1999).  A primer on heavy-tailed
  distributions.  {\it Queueing Systems}, {\bf 33}, 261--275.

\bibitem{Ver} \textsc{Veraverbeke, N.} (1977).  Asymptotic behavior of
  Wiener-Hopf factors of a random walk.  {\it Stoch.\ Process.\
    Appl.}, {\bf 5}, 27--37.

\end{thebibliography}
\end{document}